\documentclass[12pt]{article}

\usepackage{hyperref}
\hypersetup{
pdfpagemode={None},
pdftitle={}
pdfauthor={}
}
\usepackage{amsmath}
\usepackage{amssymb}
\usepackage{amsfonts}
\usepackage{amsopn}
\usepackage{amsthm}
\usepackage[all]{xy}

\swapnumbers
\theoremstyle{plain}
\newtheorem{thm}{Theorem}[section]

\newtheorem{cor}[thm]{Corollary}

\newtheorem{lem}[thm]{Lemma}
\newtheorem*{claim}{Claim}

\theoremstyle{definition}
\newtheorem{ntt}[thm]{}
\newtheorem{rem}[thm]{Remark}

\newcommand{\zz}{\mathbb{Z}} 
\newcommand{\E}{\mathrm{E}}
\newcommand{\F}{\mathrm{F}_4}
\newcommand{\w}{\bar{\omega}}
\newcommand{\edge}[1]{\ar@{-}[#1]}
\newcommand{\lulab}[1]{\ar@{}[l]_<<{#1}}
\newcommand{\rulab}[1]{\ar@{}[r]^<<{#1}}
\newcommand{\ldlab}[1]{\ar@{}[l]^<<{#1}}
\newcommand{\rdlab}[1]{\ar@{}[r]_<<{#1}}
\newcommand{\node}{*=0{\bullet}}
\DeclareMathOperator{\CH}{\mathrm{CH}}

\title{Zero-cycles on projective varieties
and the norm principle}
\author{Philippe Gille, Nikita Semenov\footnote{Supported partially
by DFG, project GI706/1-1.}}
\date{}

\begin{document}
\maketitle
\begin{abstract}
Using the Gille-Merkurjev norm principle we compute in a uniform way
the image of the degree map for quadrics (Springer's theorem),
for twisted forms of maximal orthogonal Grassmannians (theorem of Bayer-Fluckiger
and Lenstra), for $\E_6$- (Rost's theorem), and $\E_7$-varieties.

\end{abstract}

\section{Introduction}
Let $G$ be a simple algebraic group over a field $k$ and $X$ a projective $G$-homogeneous
variety. Consider the degree map
$$\deg\colon\CH_0(X)\to\zz.$$
The goal of the present paper is to provide a method to compute
the image of this map (see \cite{PSZ} for the computation of its kernel).

This problem has a long history starting probably with the Springer
theorem which says that an anisotropic quadratic form remains anisotropic over
odd degree field extensions. This statement is equivalent to the fact
that the image of the degree map when $X$ is an anisotropic quadric equals $2\zz$.

To stress the difficulty of the above problem note that a computation
of the degree map for the varieties of Borel subgroups of groups of type $\E_8$
implies Serre's Conjecture II for fields of
cohomological dimension $\le 2$ \cite{Gi01}.

Moreover, the above problem generalizes Serre's question,
whether the map
$$\mathrm{H}^1(k,G_0)\to\prod\mathrm{H}^1(K_i,G_0)$$ has trivial kernel
when $K_i$ are finite field extensions of $k$ such that ${\gcd[K_i:k]=1}$
and $G_0$ is a split group.

The image of the degree map is known in the following cases: $X$ is a
quadric (Springer's theorem), $X$ is a twisted form of a maximal orthogonal
Grassmannian (theorem of Bayer-Fluckiger and Lenstra \cite{BFL90}), $X$ is the
variety of Borel subgroups of a group of type $\F$, $\E_6$ (a theorem of Rost,
where cohomological invariants of Albert algebras are used), and $\E_7$
(Gille's theorem \cite{Gi97}, where the norm principle is used). Note also that
there are numerous papers of M.~Florence, R.~Parimala, B.~Totaro, and many others
concerning closely related problems.

In the present paper we apply the Gille-Merkurjev norm
principle \cite{Gi97}, \cite{Me96}, \cite{BM02} and give a uniform proof of the above
results.
Apart from this, we compute the image of the degree map for the varieties
of parabolic subgroups of type $7$ of groups $G$ of type $\E_7$ and
prove that anisotropic groups of type $\E_7$
remain anisotropic over odd degree field extensions.
Note that this
property is used in \cite[Corollary~6.10]{PS}
to relate the Rost invariant of $G$
and its isotropity.

\section{Norm principle and strategy of the proof}
\begin{ntt}
Let $k$ be a perfect field with $\mathrm{char}\,k\ne 2,3$,
$\Gamma=\mathrm{Gal}(\bar k/k)$ the absolute Galois group,
$G$ a connected reductive algebraic group over $k$,
$G'=[G,G]$ the commutator subgroup, $\Delta$ its Dynkin diagram, and $\Delta_0$ its
Tits index (see \cite{Ti65}).
\end{ntt}

\begin{ntt}[Special cocharacters]
Let $G_1$ be a reductive algebraic group over $k$ and
\begin{equation}\label{sequence}
1\to G_1\to G\overset{f}{\to}T=\mathbb{G}_m\to 1
\end{equation}
an exact sequence.
The cocharacter group $T_*$ can be canonically identified with the group
$\zz$. A cocharacter $\varphi\in T_*$ is called {\it $f$-special}, if
there is a $k$-homomorphism $g\colon\mathbb{G}_m\to G$ such that $f\circ g=\varphi$.
\end{ntt}

\begin{ntt}[Set $X(\varphi)$]\label{sec23}
Denote $Z'=G/G'$, $C$ the center of the simply connected cover of $G'$, $Z$ the
center of $G$, and $\mu$ the center of $G'$.

We can represent the homomorphism
$f$ as a composition $G\to Z'\to T$. In particular, there is the
induced homomorphism $\alpha\colon {Z'}^\Gamma_*\to T_*$ between the cocharacter groups.
The exact sequence $$1\to\mu\to Z\to Z'\to 1$$ induces
a homomorphism $\beta\colon {Z'}^\Gamma_*\to\mu(-1)^\Gamma$, and the canonical epimorphism
$C\to\mu$ induces a map $\gamma\colon C_*^\Gamma\to\mu(-1)^\Gamma$,
where $\mu(-1)$ is the Tate twist, i.e.,
$\mu(-1)=\mathrm{Hom}(\mu_n,\mu)$ for any $n$ with $\mu^n=1$.
For a cocharacter $\varphi\in T_*$ we define a subset $X(\varphi)\subset C_*^\Gamma$
as the set $\gamma^{-1}(\beta(\alpha^{-1}(\{\varphi\})))$.
\end{ntt}

\begin{ntt}[Set $\Omega(\varphi)$]
From now on we assume that the Dynkin diagram $\Delta$ has no multiple edges.
Following \cite[(5.8)]{Me96}
we identify $C_*$ and the character group $C^*$ and consider
$X(\varphi)$ as a subset of $C^*$.
Let $\w_i$ denote the $i$-th fundamental weight of the simply connected
cover of $G'$ (Enumeration of simple roots follows Bourbaki).
Define now $\Omega(\varphi)$ as the set of all subsets $\Theta\subset\Delta$
such that the elements
$\{\sum_{i\in I}\w_i\vert_C,\ I\subset\Delta\setminus\Theta \text{ a }
\Gamma\text{-orbit}\}$
generate a subgroup of $C^{*\Gamma}$ that intersects $X(\varphi)$.
\end{ntt}

\begin{ntt}[Type of a parabolic subgroup]
It is well-known that there is a bijective correspondence
between the conjugancy classes of parabolic subgroups of $G'_{\bar k}$
and the subsets of the set $\Delta$ of simple roots. 

The {\it type} of a parabolic subgroup is the corresponding subset of $\Delta$.
Under this identification the Borel subgroup has type $\emptyset$.
If $P$ is a maximal parabolic subgroup of type $\Delta\setminus\{\alpha_i\}$,
where $\alpha_i$ is the $i$-th simple root, then for
simplicity of notation we say that $P$ is of type $i$.
\end{ntt}

\begin{ntt}[Tits homomorphism]\label{tits}
Let $$\beta\colon C^{*\Gamma}\to\mathrm{Br}(k)$$ be the Tits homomorphism
for the simply connected cover of $G'$ defined in \cite{Ti71}.
In order to compute the sets $\Omega(\varphi)$ we need to know the restrictions of
the fundamental weights $\w_i$ to $C$
and their images under the Tits homomorphism.

Below we describe them for groups of type $^1{\mathrm{D}_n}$, $\E_6$, and
$\E_7$. We use graphical notation, where the algebra over a vertex $i$
of the Dynkin diagram stands for the image $\beta(\w_i\vert_C)$.
Apart from this, the restriction $\w_i\vert_C$ is trivial iff the respective
algebra is $k$.

\noindent
{\bf Type $\mathrm{D}_n$:}
A simply connected group of inner type $\mathrm{D}_n$ has the form
$\mathrm{Spin}(A,\sigma)$, where $A$ is a central simple algebra of degree $2n$
with an orthogonal involution $\sigma$ of the first kind with trivial
discriminant.

For the character group $C^*$ of the center
of $\mathrm{Spin}(A,\sigma)$ we have
 $$C^*=\{0,\chi,\chi^+,\chi^-\},$$ where $\chi$
(resp. $\chi^+$, $\chi^-$) is the restriction of the fundamental weight
$\w_1$ (resp. $\w_{n-1}$, $\w_n$) to the center.

Let $C^\pm(A,\sigma)$ be the direct summands of the Clifford algebra $C_0(A,\sigma)
=C^+(A,\sigma)\oplus C^-(A,\sigma)$. We have

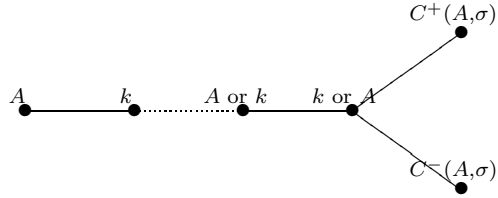
\xymatrix@=3em@R=2em{
&&&&& \node \lulab{C^+(A,\sigma)}\\
& \node \lulab{A} \edge{r} & \node \lulab{k} \ar@{..}[r] &
\node \lulab{A\text{ or }k} \edge{r}& \node \lulab{k\text{ or }A}\edge{ru}\edge{rd} & \\
&&&&& \node \lulab{C^-(A,\sigma)}
}

\medskip

We associate the Tits algebras to the last two vertices $n-1$ and $n$ in such a
way that for $\varepsilon=+$ (resp. $\varepsilon=-$) the algebra
$C^\varepsilon(A,\sigma)$ splits
over the field of rational functions of the projective homogeneous
variety of maximal parabolic subgroups of type $P_{n-1}$ (resp. $P_{n}$).
The latter are two irreducible components of the variety of
$2n^2$-dimensional isotropic right ideals $I$ of $A$ with respect to $\sigma$.

\noindent
{\bf Type $\mathrm{E}_6$:}
The Tits algebra is a certain central
simple algebra $A$ of index $1$, $3$, $9$, or $27$ and of exponent $1$
or $3$.

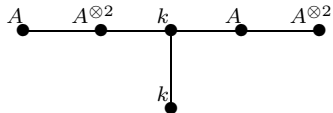
\xymatrix@=2em{
& \node \lulab{A} \edge{r}& \node \lulab{A^{\otimes 2}} \edge{r} & 
\node \lulab{k} \edge{r}\edge{d} & \node \lulab{A}
\edge{r} & \node \lulab{A^{\otimes 2}}\\
&&& \node \lulab{k} &&
}

\noindent
{\bf Type $\mathrm{E}_7$:}
The Tits algebra is a certain central
simple algebra $A$ of index $1$, $2$, $4$, or $8$ and of exponent $1$
or $2$.

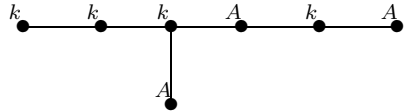
\xymatrix@=2em{
&\node\lulab{k}\edge{r}& \node \lulab{k} \edge{r}& \node \lulab{k} \edge{r}\edge{d} & 
\node \lulab{A} \edge{r} & \node \lulab{k}
\edge{r} & \node \lulab{A}\\
&&& \node \lulab{A} &&&
}
\end{ntt}

Under the above assumptions the following lemmas hold:

\begin{lem}[{\cite[Lemma~3.4]{Me96}}]\label{l1}
Let $K/k$ be a finite field extension lying in the algebraic closure $\bar k$
and let $\varphi\in T_*$. If the cocharacter $\varphi$ is $f_K$-special,
then the cocharacter $[K:k]\varphi$ is $f$-special.
\end{lem}

\begin{lem}[{\cite[Theorem~5.6]{Me96}}]\label{l3}
For a cocharacter $\varphi\in T_*$ the following conditions are equivalent:
\begin{enumerate}
\item $\varphi$ is $f$-special;
\item there exists a parabolic subgroup of $G$ defined over $k$ whose
type is contained in $\Omega(\varphi)$.
\end{enumerate}
\end{lem}

\begin{lem}[{\cite[Proposition~5.8]{Me96}}]\label{l4}
Let $\beta\colon C^{*\Gamma}\to\mathrm{Br}(k)$ be the Tits homomorphism
for the simply connected cover of $G'$. Assume that the
Dynkin diagram $\Delta$ has no multiple edges.
If a cocharacter $\varphi\in T_*$ is $f$-special, then $0\in\beta(X(\varphi))$.
\end{lem}

\begin{thm}\label{main}
Let $X$ be an anisotropic smooth projective variety over $k$ and $p$ a prime number.
In the above notation assume that the following conditions hold:
\begin{enumerate}
\item For any field extension $K/k$ and for any coprime to $p$
cocharacter $\varphi$, if $0\in\beta_K(X(\varphi))\subset\mathrm{Br}(K)$ and
$G'$ has a parabolic subgroup defined
over $K$ whose type is contained in $\Omega(\varphi)$, then $X(K)\ne\emptyset$;
\item For any field extension $K/k$ and for any coprime to $p$
cocharacter $\varphi$ if $X(K)\ne\emptyset$,
then there exists a parabolic subgroup of $G'$ of type contained in
$\Omega(\varphi)$ defined over $K$.
\end{enumerate}
Then $\deg(\CH_0(X))\subset p\zz.$
\end{thm}
\begin{proof}
Let $K/k$ be a field extension. We show first the following
\begin{claim}
$X(K)\ne\emptyset$ if and only if any coprime to $p$ cocharacter $\varphi$
is $f_K$-special.
\end{claim}
Indeed, if $X(K)\ne\emptyset$, then by item~2 there is a parabolic subgroup
of $G'$ defined over $K$ whose type is contained in $\Omega(\varphi)$.
By Lemma~\ref{l3} $\varphi$ is $f_K$-special.

Conversely, if $\varphi$ is $f_K$-special, then
by Lemma~\ref{l4} we have $0\in\beta_K(X(\varphi))$,
and by Lemma~\ref{l3} there is a parabolic subgroup of $G'$ defined over $K$
of type contained in $\Omega(\varphi)$.
Therefore by item~1 we have $X(K)\ne\emptyset$.

Let now $K/k$ be a finite field extension such that $X(K)\ne
\emptyset$. To finish the proof of the theorem it sufficies to show
that $[K:k]$ is divisible by $p$. Assume the converse.

Since $X(K)\ne\emptyset$, by Claim any coprime to $p$ cocharacter
$\varphi$ is $f_K$-special. By Lemma~\ref{l1} the cocharacter
$[K:k]\varphi$ is $f$-special. Therefore by Claim $X(k)\ne\emptyset$.
Contradiction.
\end{proof}

\section{Applications}

\begin{cor}[Springer's theorem]
Let $A$ be a central simple $k$-algebra of degree $2n\ge 4$ with 
an orthogonal involution $\sigma$ of the first kind.
Let $X$ be the variety of isotropic with respect to $\sigma$ right ideals
of $A$ of dimension $2n$. Assume $X$ is anisotropic.

Then $\deg(\CH_0(X))\subset 2\zz$. In particular, if $X$ is an
anisotropic smooth even-dimensional projective quadric, then $\deg(\CH_0(X))=2\zz$.
\end{cor}
\begin{proof}(\cite[Lemma~6.2]{Me96}).
There is the following exact sequence of groups:
$$1\to G_1=\mathrm{Spin}(A,\sigma)\to G=\Gamma(A,\sigma)\overset{f}{\to} \mathbb{G}_m\to 1,$$
where $\Gamma(A,\sigma)$ is the Clifford group and $f$ is the {\it spinor norm} homomorphism.

Let $p=2$. It is easy to check that for any odd cocharacter $\varphi$
the set $X(\varphi)=\{\chi\}$, where $\chi$ is the restriction of $\w_1$
to the center $C$.

Let $K/k$ be a field extension. If $0\in\beta_K(X(\varphi))$, then
the algebra $A_K$ is split (see \ref{tits}). Thus, $\sigma_K$
corresponds to a quadratic form, and $X_K$ is a projective quadric.
If additionally $G'_K$ has a parabolic subgroup defined over $K$
of type contained in $\Omega(\varphi)$, then it easy to see that
this quadratic form is isotropic, and thus $X(K)\ne\emptyset$.

Finally, if $X(K)\ne\emptyset$, then
$G'$ has a parabolic subgroup of type $\Delta\setminus\{\alpha_1\}$,
where $\alpha_1$ is the first simple root. But $\Delta\setminus\{\alpha_1\}
\in\Omega(\varphi)$.

Thus, we checked all conditions of Theorem~\ref{main}.
\end{proof}

\begin{cor}[Bayer-Fluckiger and Lenstra]
Let $A$ be a central simple algebra of degree $2n\ge 4$ with
an orthogonal involution $\sigma$ of the first kind.
Let $Y$ be the variety
of $2n^2$-dimensional isotropic right ideals of $A$ and
$$X=Y\times\mathrm{Spec}(k[t]/(t^2-\mathrm{disc}(\sigma)).$$
Assume $X$ is anisotropic.
Then $\deg(\CH_0(X))\subset 2\zz.$
\end{cor}
\begin{proof}(\cite[6.3]{Me96}).
Consider the following exact sequence of groups:
$$1\to G_1=\mathrm{O}^+(A,\sigma)\to G=\mathrm{GO}^+(A,\sigma)
\overset{f}{\to}\mathbb{G}_m\to 1,$$
where $f$ is the {\it multiplier map}.

Let $p=2$. Denote as $\chi^+$ (resp. $\chi^-$) the
restriction of the fundamental weight $\w_{n-1}$ (resp. $\w_n$)
to the center.
It is easy to check that for any odd cocharacter $\varphi$ we have
$$X(\varphi)=\begin{cases}
\emptyset, & \mathrm{disc}(\sigma)\ne 1;\\
\{\chi^+,\chi^-\}, & \mathrm{disc}(\sigma)=1.
\end{cases}$$

Finally, if $0\in\beta_K(X(\varphi))$, then $\mathrm{disc}(\sigma_K)=1$.
Then the variety $Y_K$ is the disjoint union of varieties
of parabolic subgroups of $G'_K$ of types $\Delta\setminus\{\alpha_{n-1}\}$
and $\Delta\setminus\{\alpha_n\}$.
If additinally $G'$ has a parabolic subgroup defined over $K$
of type contained in $\Omega(\varphi)$, then the Tits index $\Delta_0$
of $G'_K$ contains at most one of the roots $\alpha_{n-1}$, $\alpha_n$
(see \ref{tits}).
Therefore in this case $X(K)\ne\emptyset$.

To finish the proof of the corollary it remains to notice that condition~2
of Theorem~\ref{main} is obvious.
\end{proof}

Using similar arguments one can show the following well-known statement.
Opposite to the traditional approach our proof does not use cohomological
invariants of Albert algebras.

\begin{cor}[M.~Rost]
Let $G_1$ be a simply connected algebraic group of type $^1{\E_6}$
over $k$ and $X$ the variety of its parabolic subgroups
of type $1$ (resp. $6$). Assume $X$ is anisotropic.
Then $\CH_0(X)\subset 3\zz$.
\end{cor}
\begin{proof}
If $G_1$ has a non-trivial Tits algebra, then the statement is obvious,
since for a field extension $K/k$ condition $X(K)\ne\emptyset$ implies
that the Tits algebras of $(G_1)_K$ are split.

Assume that $G_1$ has trivial Tits algebras.
Let $J$ be an Albert algebra associated with $G_1$ (see \cite{Ga01a}).
A $k$-linear map $$h\colon J\to J$$ is called
a {\it similarity} if there exists $\alpha_h\in k^\times$ (the multiplier of $h$)
such that $$N(h(j))=\alpha_h N(j)$$ for all $j\in J$, where $N$ stands
for the cubic norm on $J$.
Then $G_1$ coincides with the similarities of this Jordan algebra
with multiplier $1$. Let $G$ be the group of all similarities.
Then $G$ is a reductive group and there is the following
exact sequence of algebraic groups:
$$1\to G_1\to G\overset{f}{\to} T=\mathbb{G}_m\to 1,$$
where the map $f$ is defined on $k$-points as $h\mapsto\alpha_h$.

Let $p=3$ and let $\varphi\in T_*=\zz$ be a cocharacter coprime to $3$.
We check now the conditions of Theorem~\ref{main}.

First we compute $X(\varphi)$. In our situation $G'=[G,G]=G_1$,
$Z'=T=\mathbb{G}_m$, $\mu=\mu_3$, $C=\mu_3$, $Z=\mathbb{G}_m$,
and the group
$$C_*\simeq C^*=\zz/3=\{0,\w_1\vert_C,\w_6\vert_C=-\w_1\vert_C\},$$
where $\w_i\vert_C$ denotes the restriction of the $i$-th fundamental
weight of $G_1$ to the center, $i=1,6$.
Therefore, $X(\varphi)=\{\w_1\vert_C\}$ or
$X(\varphi)=\{\w_6\vert_C\}$ (it depends on $\varphi$ mod $3$).

Let $K/k$ be a field extension. Assume first that $G'_K$ is isotropic
and the type of a parabolic subgroup $P$ of $G'$ defined over $K$
is contained in $\Omega(\varphi)$.
If the parabolic subgroup of $G'$ of type $1$ is not defined, then
by Tits's classification \cite[p.~58]{Ti65}
the Tits index of $G'$ equals $\Delta_0=\Delta\setminus\{\alpha_2,\alpha_4\}$.
But the restrictions
to the center of the $2$-nd and of the $4$-th fundamental weights are trivial
(see \cite[p.~653]{Ti90} or \ref{tits}). This contradicts to the assumption that
the type of $P$ is contained in $\Omega(\varphi)$.

Finally, condition~2 of Theorem~\ref{main} is obvious.
\end{proof}

\begin{rem}
If the Tits algebras of $G_1$ are trivial, then the image of the
degree homomorphism $\CH_0(X)\to\zz$ equals $3\zz$.
\end{rem}

\begin{cor}
Let $G_1$ be a simply connected algebraic group of
type $\E_7$ over $k$ and $X$ the variety of
maximal parabolic subgroup of $G_1$ of type $7$. Assume $X$ is anisotropic.
Then $\CH_0(X)\subset 2\zz$.
\end{cor}
\begin{proof}
Let $(A,\sigma,\pi)$, where $A$ is a central simple
$k$-algebra with a symplectic involution $\sigma$ and
$\pi\colon A\to A$ a linear map,
be {\it a gift} associated with $G_1$ (see \cite{Fe72}, \cite{Ga01a} and \cite{Ga01b}).
An invertible element $h\in A$ is called
a {\it similarity} if there exists $\alpha_h\in k^\times$ (the multiplier of $h$)
such that $$\sigma(h)h=\alpha_h\cdot 1$$ and
$$\pi(hah^{-1})=\alpha_h h\pi(a)h^{-1}$$ for all $a\in A$.
Then $G_1$ coincides with the similarities of this gift
with multiplier $1$. Let $G$ be the group of all similarities.
Then $G$ is a connected reductive group and there is the following
exact sequence of algebraic groups:
$$1\to G_1\to G\overset{f}{\to} T=\mathbb{G}_m\to 1,$$
where the map $f$ is defined on $k$-points as $h\mapsto\alpha_h$.

Let $p=2$ and let $\varphi$ be an odd cocharacter.

First we compute $X(\varphi)$. In our situation $G'=[G,G]=G_1$,
$Z'=T=\mathbb{G}_m$, $\mu=\mu_2$, $C=\mu_2$, and $Z=\mathbb{G}_m$.
Thus, the maps $\alpha$ and $\gamma$ from \ref{sec23} are
the identity maps. The map $\beta\colon Z'_*=\zz\to\mu_2(-1)=\zz/2$
from \ref{sec23} is the
usual factor-map. Therefore, $X(\varphi)=\{\chi\}$, where
as $\chi$ we denote a unique non-trivial element of $C_*\simeq C^*$.

Let $K/k$ be a field extension. Assume that $0\in\beta_K(X(\varphi))$,
$G'_K$ is isotropic and the type of a parabolic subgroup of $G'$
defined over $K$ is contained in $\Omega(\varphi)$.
The first assumption implies that the Tits algebra $A$ of $G_1$ is split.

If the parabolic subgroup of $G'$ of type $7$ is not defined, then
by Tits's classification \cite[p.~59]{Ti65}
the Tits index of $G'$ equals $\Delta_0=\Delta\setminus\{\alpha_1\}$.
But the restriction
to the center of the $1$-st fundamental weight is trivial
(see \cite[p.~653]{Ti90}).
Therefore we have $X(K)\ne\emptyset$.

Finally, condition~2 of Theorem~\ref{main} is obvious.
\end{proof}

\begin{rem}
If the Tits algebras of $G_1$ are trivial, then the image of the
degree homomophism $\CH_0(X)\to\zz$ equals $2\zz$.
\end{rem}

\begin{cor}
A group $G_1$ as in the statement of the previous corollary does not
split over an odd degree field extension.
\end{cor}

\bibliographystyle{chicago}

\medskip

\noindent
{\sc UMR 8553 du CNRS,
\'Ecole Normale Sup\'erieure,
45 Rue d'Ulm,
75005 Paris,
France}

\medskip

\noindent
{\sc Mathematisches Institut,
Universit\"at M\"unchen,
Theresienstr. 39,
80333 M\"unchen,
Germany}
\end{document}